\documentclass[12pt]{article} 
\usepackage{amsfonts,amsmath,latexsym,amssymb,mathrsfs,amsthm}

\def\numberlikeadb{\global\def\theequation{\thesection.\arabic{equation}}}
\numberlikeadb
\newtheorem{theorem}{Theorem}[section]
\newtheorem{lemma}[theorem]{Lemma}
\newtheorem{corollary}[theorem]{Corollary}

\newtheorem{proposition}[theorem]{Proposition}
\newtheorem{remark}[theorem]{Remark}

\numberwithin{equation}{section}

\begin{document}

\title{Rates of convergence in normal approximation under moment conditions via new bounds on solutions of the Stein equation}
\author{Robert E. Gaunt\footnote{Department of Statistics,
University of Oxford, 1 South Parks Road, OXFORD OX1 3TG, UK; supported by EPSRC research grant AMRYO100. }
}

\date{June 2014} 
\maketitle

\begin{abstract}New bounds for the $k$-th order derivatives of the solutions of the normal and multivariate normal Stein equations are obtained.  Our general order bounds involve fewer derivatives of the test function than those in the existing literature.  We apply these bounds and local approach couplings to obtain an order $n^{-(p-1)/2}$ bound, for smooth test functions, for the distance between the distribution of a standardised sum of independent and identically distributed random variables and the standard normal distribution when the first $p$ moments of these distributions agree.  We also obtain a bound on the convergence rate of a sequence of distributions to the normal distribution when the moment sequence converges to normal moments. 
\end{abstract}


\noindent{{{\bf{AMS 2010 Subject Classification:}}} 60F05

\section{Introduction}

In 1972, Stein \cite{stein} introduced a powerful method for assessing the distance between a probability distribution and the normal distribution. Central to the technique is the following inhomogeneous differential equation, known as the Stein equation: 
\begin{equation} \label{normal equation} f''(w)-wf'(w)=h(w)-\Phi h,
\end{equation} 
where $\Phi h$ denotes the quantity $\mathbb{E}h(Z)$ for $Z\sim N(0,1)$, and the test function $h$ is real-valued.  Evaluating both sides of (\ref{normal equation}) at a random variable $W$ and taking expectations gives
\begin{equation} \label{expect} \mathbb{E}[f''(W)-Wf'(W)]=\mathbb{E}h(W)-\Phi h.
\end{equation}
Thus, we can bound the quantity $\mathbb{E}h(W)-\Phi h$ by solving the Stein equation (\ref{normal equation}) and then bounding the left-hand side of (\ref{expect}).  

Typically, the left-hand side of (\ref{expect}) is bounded by Taylor expanding about a random variable coupled with $W$.  As a result of the Taylor expansions, we often require bounds on at the least the first three derivatives of the solution of the Stein equation (\ref{normal equation}).  Stein \cite{stein2} showed that when the test function $h$ is bounded and absolutely continuous the unique bounded solution of the Stein equation (\ref{normal equation}) is given by
\begin{eqnarray} \label{hundred} f'(w)&=&-\mathrm{e}^{w^2/2}\int_{w}^{\infty}[h(t)-\Phi h]\mathrm{e}^{-t^2/2}\,\mathrm{d}t \\
\label{hundred1}&=&\mathrm{e}^{w^2/2}\int_{-\infty}^w[h(t)-\Phi h]\mathrm{e}^{-t^2/2}\,\mathrm{d}t,
\end{eqnarray}
and that the following bounds on its derivatives hold
\begin{equation}\label{doblerbound} \|f'\| \leq \sqrt{\frac{\pi}{2}}\|h-\Phi h\|, \qquad \|f''\| \leq 2\|h-\Phi h\|, \qquad \|f^{(3)}\| \leq 2\|h'\|,
\end{equation}
where $\|f\|:=\|f\|_{\infty}=\sup_{x\in\mathbb{R}}|f(x)|$.

Over the years, Stein's method has been extended to a variety of distributions; for an overview see Reinert \cite{reinert 0}.  Stein's method was adapted to the multivariate normal distribution by Barbour \cite{barbour2} and G\"{o}tze \cite{gotze}.  Barbour and G\"{o}tze recognised the left-hand side of the Stein equation (\ref{normal equation}) as the generator of an Ornstein-Uhlenbeck process and used the theory of generators of stochastic processes to solve the Stein equation and bound the derivatives of the solution.  

Let $\Sigma$ be a $d\times d$ positive-definite matrix, then a Stein equation for the multivariate normal distribution with mean $\mathbf{0}$ and covariance matrix $\Sigma$ (see Goldstein and Rinott \cite{goldstein1}) is given by
\begin{equation} \label{mvn1} \nabla^T\Sigma\nabla f(\mathbf{w})-\mathbf{w}^T\nabla f(\mathbf{w})=h(\mathbf{w})-\mathbb{E}h(\Sigma^{1/2}\mathbf{Z}),
\end{equation}
where $\mathbf{Z}$ denotes a random vector having standard multivariate normal distribution of dimension $d$.  The solution of (\ref{mvn1}) is given by
\begin{equation}\label{mvnsoln}f(\mathbf{w})=-\int_{0}^{\infty}[\mathbb{E}h(\mathrm{e}^{-s}\mathbf{w}+\sqrt{1-\mathrm{e}^{-2s}}\Sigma^{1/2}\mathbf{Z})-\mathbb{E}h(\Sigma^{1/2}\mathbf{Z})]\,\mathrm{d}s.
\end{equation}
In the univariate case, the solution (\ref{hundred}) can be seen to be the first derivative of (\ref{mvnsoln}).  However, the representation (\ref{mvnsoln}) of the solution leads to much simpler calculations of derivatives of the solution of the normal Stein equation.  Indeed, if $\frac{\partial^k h(x)}{\prod_{j=1}^k\partial x_{i_j}}$ is bounded, then, by dominated convergence,
\begin{equation}\label{3x3min}\frac{\partial^k f(\mathbf{w})}{\prod_{j=1}^k\partial w_{i_j}}=-\int_{0}^{\infty}\mathrm{e}^{-ks}\mathbb{E}\bigg[\frac{\partial^k h}{\prod_{j=1}^k\partial w_{i_j}}(\mathrm{e}^{-s}\mathbf{w}+\sqrt{1-\mathrm{e}^{-2s}}\Sigma^{1/2}\mathbf{Z})\bigg]\mathrm{d}s.
\end{equation}
Taking absolute values and using that $\int_0^{\infty}\mathrm{e}^{-ks}\,\mathrm{d}s=\frac{1}{k}$ yields the following bound (see Barbour \cite{barbour2}, Goldstein and Rinott \cite{goldstein1} and Reinert and R\"{o}llin \cite{reinert 1}):
\begin{equation}\label{cheque} \bigg\|\frac{\partial^k f(\mathbf{w})}{\prod_{j=1}^k\partial w_{i_j}}\bigg\|\leq \frac{1}{k}\bigg\|\frac{\partial^k h(\mathbf{w})}{\prod_{j=1}^k\partial w_{i_j}}\bigg\|, \qquad k\geq 1.
\end{equation}
Meckes \cite{meckes} obtained the following bound for the $k$-th derivative of $f$ as a $k$-linear form:
\begin{equation}\label{meckesk}M_k(f)\leq\frac{1}{k}M_k(h), \quad k\geq 1,
\end{equation}
where $M_k(f):=\sup_{\mathbf{w}\in\mathbb{R}^d}\|D^kf(\mathbf{w})\|_{op}$ and $\|D^kf(\mathbf{w})\|_{op}$ is the operator norm of the $k$-th derivative of $f$ as a $k$-linear form.  The bound (\ref{meckesk}) is coordinate-free and thus could be viewed as a geometrically more natural bound than (\ref{cheque}).  Moreover, bounds given in terms of $M_k(h)$  are typically better than those give in terms partial derivatives of $h$, especially in high dimensions; see Meckes \cite{meckes} for a more detailed discussion.

In this paper, we establish new bounds for derivatives of general order of the solutions of the normal Stein equation (\ref{normal equation}) and the multivariate normal Stein equation (\ref{mvn1}).  Firstly, we consider the multivariate case and obtain bounds (given in Proposition \ref{fingerscrossed}) that improve on (\ref{cheque}) and (\ref{meckesk}) by involving one fewer derivative of the test function $h$, thereby allowing us to impose weaker differentiability conditions on $h$.  Chatterjee and Meckes \cite{chatterjee 3} have also obtained bounds for the first three derivatives of the solution of the multivariate normal Stein equation which have this smoothing property.  We then specialise to the univariate case and obtain a bound (Proposition \ref{snnewz}) that involves two fewer derivatives than (\ref{cheque}).  In a sense, our bound can be thought of as a generalisation of Stein's classic bound $\|f^{(3)}\|\leq 2\|h'\|$ to higher order derivatives of $f$.

Often, Kolmogorov or Wasserstein distances are of interest, and, in these cases, our bounds for the solution, which involve higher order derivatives of $h$, would not be suitable.  However, the strength of these bounds lies in higher-order asymptotics: faster than order $n^{-1/2}$ convergence rates under moment conditions. 

We begin Section 3 by applying our new bounds for the derivatives of the solution of standard normal Stein equation and local approach couplings to obtain a general bound for the distance between the distribution of a standardised sum of independent random variables and the standard normal distribution (Theorem \ref{123421}).  We then apply this theorem to obtain a bound on the distance between the distribution of a standardised sum of independent random variables and the standard normal distribution when the first $p$ moments of these distributions are equal (Corollary \ref{1234321ww}).  When we specialise to the case of identically distributed random variables, we have an order $n^{-(p-1)/2}$ convergence rate for smooth test functions, thereby extending a result of Goldstein and Reinert \cite{goldstein} in which they obtained a convergence rate of order $n^{-1}$ when the third moment is zero.  

Our result appears to be new, however there are other instances in the literature of convergence rates faster than order $n^{-1}$ for normal approximation in the case of matching moments; for example, Edgeworth exapnsions (see Hall \cite{hall}).  As was observed by Goldstein and Reinert \cite{goldstein} pp$.$ 937--938, the bounds arrived at through Edgeworth expansions depend on the smoothness of the distribution $F$ we are approximating, whereas we show that for smooth test functions, bounds of order $n^{-(p-1)/2}$ hold for any $F$ with first $p$ moments agreeing with those of the standard normal distribution even if F does not possess a density.

In Theorem \ref{billley}, we obtain a bound on the convergence rate of a sequence of distributions to the standard normal distribution when the moment sequence convergences to standard normal moments.  We end this paper by considering a straightforward generalisation of Theorem \ref{123421} to higher dimensions.

\vspace{3mm}

\emph{Notation.} Throughout this paper we shall denote supremum norm of a real-valued function $f$ by $\|f\|:=\|f\|_{\infty}=\sup_{x\in\mathbb{R}}|f(x)|$.  In $\mathbb{R}^d$, the Euclidean inner product is denoted $\langle\cdot,\cdot \rangle$ and the Euclidean norm is denoted
$|\cdot|$. The operator norm of a matrix $A$ over $\mathbb{R}$ is defined by
\[\|A\|_{op}:=\sup_{|\mathbf{u}|=1}|A\mathbf{u}|=\sup_{\mathbf{u}\not=0}\frac{|A\mathbf{u}|}{|\mathbf{u}|}.\]
More generally, if $A$ is a $k$-linear form on $\mathbb{R}^d$, the operator norm of $A$ is defined to be
\[\|A\|_{op}:=\sup\{|A(\mathbf{u}_1,\ldots,\mathbf{u}_k)|\,:\,|\mathbf{u}_1|=\cdots=|\mathbf{u}_k|=1\}.\]
We shall write $C^n(\mathbb{R}^d)$ for the space of $k$ times differentiable functions on $\mathbb{R}^d$.  We let $C_b^n(\mathbb{R}^d)$ denote the space of bounded functions on $\mathbb{R}^d$ with bounded $k$-th order derivatives for $k\leq n$, and we let $C_b^{\infty}(\mathbb{R})$ denote the space of bounded real-valued functions with all derivatives bounded.  The $k$-th derivative $D^kf(x)$ of a function $f\in C^k(\mathbb{R}^d)$ is a $k$-linear
form on $\mathbb{R}^d$, given in coordinates by
\[D^kf(\mathbf{w})(\mathbf{u}_1,\ldots,\mathbf{u}_k):=\sum_{i_1,\ldots,i_k=1}^d\frac{\partial^kf}{\partial x_{i_1}\ldots\partial x_{i_k}}(\mathbf{w})(\mathbf{u}_1)_{i_1}\ldots(\mathbf{u}_k)_{i_k},\]
where $(\mathbf{u}_i)_j$ denotes the $j$-th component of the vector $\mathbf{u}_i$.  For $f\in C_b^k(\mathbb{R}^d)$, we let
\[M_k(f):=\sup_{\mathbf{w}\in\mathbb{R}^d}\|D^kf(\mathbf{w})\|_{op}.\]
Lastly, $Z$ will denote a random variable with standard normal distribution and $\mathbf{Z}$ denotes a random vector having standard multivariate normal distribution of dimension $d$.

\section{Bounds for derivatives of the solutions of normal and multivariate normal Stein equations} 
We now present our bounds for the $k$-th order derivatives of the solution (\ref{mvnsoln}) of the multivariate normal Stein equation (\ref{mvn1}).  To arrive at a bound involving one fewer derivative of $h$ than $f$, we use an argument involving integration by parts which is very similar to that of Chatterjee and Meckes \cite{chatterjee 3} and Meckes \cite{meckes}.

\begin{proposition}\label{fingerscrossed}Suppose $\Sigma$ is a $d\times d$ positive-definite matrix and that $h$ is bounded.  Then the first order partial derivatives the solution (\ref{mvnsoln}) of the multivariate normal Stein equation (\ref{mvn1}) are bounded by
\begin{equation}\label{usthem}\bigg\|\frac{\partial f(\mathbf{w})}{\partial w_i}\bigg\|\leq \sqrt{\frac{\pi}{2}}\Bigg[\sum_{j=1}^d\tilde{\sigma}_{ij}^2\Bigg]^{1/2}\|h-\mathbb{E}h(\Sigma^{1/2}\mathbf{Z})\|,
\end{equation}
where $\tilde{\sigma}_{ij}=(\Sigma^{-1/2})_{ij}$.  Suppose now that $h\in C_b^{n-1}(\mathbb{R}^d)$, where $n\geq 2$.  Then
\begin{equation}\label{charm}\bigg\|\frac{\partial^k f(\mathbf{w})}{\prod_{j=1}^k\partial w_{i_j}}\bigg\|\leq\frac{\Gamma(\frac{k}{2})}{\sqrt{2}\Gamma(\frac{k+1}{2})}\min_{1\leq l\leq k}\Bigg\{\Bigg[\sum_{j=1}^d\tilde{\sigma}_{i_lj}^2\Bigg]^{1/2} \bigg\|\frac{\partial^{k-1} h(\mathbf{w})}{\prod_{\stackrel{1\leq j\leq k}{j\not=l}}\partial w_{i_j}}\bigg\|\Bigg\}, \quad k=2,\ldots,n.
\end{equation}
 
With the same assumptions on $\Sigma$ and $h$ as before, we have the following bounds for the operator norm of
the $k$-th derivative of $f$ as a $k$-linear form:
\begin{align}\label{m1f}M_1(f)&\leq\sqrt{\frac{\pi}{2}}\|\Sigma^{-1/2}\|_{op}\|h-\mathbb{E}h(\Sigma^{1/2}\mathbf{Z})\|,\\
\label{mkf}M_k(f)&\leq\frac{\Gamma(\frac{k}{2})}{\sqrt{2}\Gamma(\frac{k+1}{2})}\|\Sigma^{-1/2}\|_{op}M_{k-1}(h), \quad k=2,\ldots, n.
\end{align}
\end{proposition}

\begin{proof}We begin by obtaining a formula for the $k$-th order partial derivatives of the solution of the multivariate normal Stein equation.  From (\ref{mvnsoln}) we have
\begin{equation*}f(\mathbf{w})=-\int_{0}^{\infty}\!\int_{\mathbb{R}^d}[h(\mathrm{e}^{-s}\mathbf{w}+\sqrt{1-\mathrm{e}^{-2s}}\mathbf{x})-\mathbb{E}h(\Sigma^{1/2}\mathbf{Z})]p(\mathbf{x})\,\mathrm{d}\mathrm{x}\,\mathrm{d}s,
\end{equation*}
where 
\[p(\mathbf{x})=\frac{1}{(2\pi)^{d/2}\sqrt{\det(\Sigma)}}\exp\bigg(-\frac{1}{2}\mathbf{x}^T\Sigma^{-1}\mathbf{x}\bigg).\]
Making the substitution $\mathbf{y}=\mathrm{e}^{-s}\mathbf{w}+\sqrt{1-\mathrm{e}^{-2s}}\mathbf{x}$ gives
\begin{equation*}
f(\mathbf{w})=-\int_0^{\infty}\!\int_{\mathbb{R}^d}\frac{1}{(1-\mathrm{e}^{-2s})^{d/2}}[h(\mathbf{y})-\mathbb{E}h(\Sigma^{1/2}\mathbf{Z})]p\bigg(\frac{\mathbf{y}-\mathrm{e}^{-s}\mathbf{w}}{\sqrt{1-\mathrm{e}^{-2s}}}\bigg)\,\mathrm{d}\mathbf{y}\,\mathrm{d}s.
\end{equation*}
We now note that
\[\frac{\partial}{\partial x_{i}}(\mathbf{x}^T\Sigma^{-1}\mathbf{x})=2(\Sigma^{-1}\mathbf{x})_i.\]
Therefore, by dominated convergence, since $h$ is bounded, we have
\begin{align}\frac{\partial f(\mathbf{w})}{\partial w_{i}}&=-\int_0^{\infty}\!\int_{\mathbb{R}^d}\frac{\mathrm{e}^{-s}}{(1-\mathrm{e}^{-2s})^{(d+1)/2}}(\Sigma^{-1}(\mathrm{y}-\mathrm{e}^{-s}\mathrm{w}))_i[h(\mathbf{y})-\mathbb{E}h(\Sigma^{1/2}\mathbf{Z})]\nonumber\\
&\quad\times p\bigg(\frac{\mathbf{y}-\mathrm{e}^{-s}\mathbf{w}}{\sqrt{1-\mathrm{e}^{-2s}}}\bigg)\,\mathrm{d}\mathbf{y}\,\mathrm{d}s\nonumber\\
&=-\int_0^{\infty}\!\int_{\mathbb{R}^d}\frac{\mathrm{e}^{-s}}{\sqrt{1-\mathrm{e}^{-2s}}}(\Sigma^{-1}\mathbf{x})_i[h(\mathrm{e}^{-s}\mathbf{w}+\sqrt{1-\mathrm{e}^{-2s}}\mathbf{x})-\mathbb{E}h(\Sigma^{1/2}\mathbf{Z})] p(\mathbf{x})\,\mathrm{d}\mathbf{x}\,\mathrm{d}s\nonumber\\
&=-\int_0^{\infty}\frac{\mathrm{e}^{-s}}{\sqrt{1-\mathrm{e}^{-2s}}}\mathbb{E}\Big[(\Sigma^{-1/2}\mathbf{Z})_i[h(\mathrm{e}^{-s}\mathbf{w}+\sqrt{1-\mathrm{e}^{-2s}}\Sigma^{1/2}\mathbf{Z})-\mathbb{E}h(\Sigma^{1/2}\mathbf{Z})]\Big]\,\mathrm{d}s.
\end{align}
If $\frac{\partial^{k-1} h(\mathbf{w})}{\prod_{j\not= l}^k\partial w_{i_j}}$ is bounded, then by dominated convergence we have, for any $l\in\{1,\ldots,k\}$, 
\begin{equation}\label{ff6ffvi}\frac{\partial^k f(\mathbf{w})}{\prod_{j=1}^k\partial w_{i_j}}=-\int_{0}^{\infty}\frac{\mathrm{e}^{-ks}}{\sqrt{1-\mathrm{e}^{-2s}}}\mathbb{E}\bigg[(\Sigma^{-1/2}\mathbf{Z})_{i_l}\frac{\partial^{k-1} h}{\prod_{\stackrel{1\leq j\leq k}{j\not=l}}\partial w_{i_j}}(\mathrm{e}^{-s}\mathbf{w}+\sqrt{1-\mathrm{e}^{-2s}}\Sigma^{1/2}\mathbf{Z})\bigg]\,\mathrm{d}s.
\end{equation}
Therefore
\begin{equation*}\bigg\|\frac{\partial f(\mathbf{w})}{\partial w_{i}}\bigg\|\leq\mathbb{E}|(\Sigma^{-1/2}\mathbf{Z})_i|\|h-\mathbb{E}h(\Sigma^{1/2}\mathbf{Z})\|\int_0^{\infty}\frac{\mathrm{e}^{-s}}{\sqrt{1-\mathrm{e}^{-2s}}}\,\mathrm{d}s
\end{equation*}
and
\[\bigg\|\frac{\partial^k f(\mathbf{w})}{\prod_{j=1}^k\partial w_{i_j}}\bigg\|\leq\min_{1\leq l\leq k}\bigg\{\mathbb{E}|(\Sigma^{-1/2}\mathbf{Z})_{i_l}|\bigg\|\frac{\partial^{k-1} h(\mathbf{w})}{\prod_{\stackrel{1\leq j\leq k}{j\not=l}}\partial w_{i_j}}\bigg\|\bigg\}\int_0^{\infty}\frac{\mathrm{e}^{-ks}}{\sqrt{1-\mathrm{e}^{-2s}}}\,\mathrm{d}s.\]
Now, $(\Sigma^{-1/2}\mathbf{Z})_i\sim N(0,\sum_{j=1}^d\tilde{\sigma}_{ij}^2)$, and so $\mathbb{E}|(\Sigma^{-1/2}\mathbf{Z})_i|=\sqrt{\frac{2}{\pi}}\sqrt{\sum_{j=1}^d\tilde{\sigma}_{ij}^2}$.  We also have the definite integral formula
\begin{equation}\label{intgamma}\int_0^{\infty}\frac{\mathrm{e}^{-ks}}{\sqrt{1-\mathrm{e}^{-2s}}}\,\mathrm{d}s=\frac{1}{2}\int_0^{1}(1-t)^{-1/2}t^{k/2-1}\,\mathrm{d}t=\frac{1}{2}B\bigg(\frac{1}{2},\frac{k}{2}\bigg)=\frac{\Gamma(\frac{1}{2})\Gamma(\frac{k}{2})}{2\Gamma(\frac{k+1}{2})}=\frac{\sqrt{\pi}\Gamma(\frac{k}{2})}{2\Gamma(\frac{k+1}{2})},
\end{equation}  
where $B(a,b)$ is the beta function.  Inequalities (\ref{usthem}) and (\ref{charm}) now follow.

We end by proving inequality (\ref{mkf}); the proof of inequality (\ref{m1f}) is very similar.  Using (\ref{ff6ffvi}) we have that, for unit vectors $\mathbf{u}_1,\ldots,\mathbf{u}_k$, and any $l\in\{1,\ldots,k\}$,
\begin{align*}&|D^kf(\mathbf{w})(\mathbf{u}_1,\ldots,\mathbf{u}_k)|\\
&=\Bigg|\sum_{1\leq i_1<\cdots<i_k\leq d}\int_{0}^{\infty}\frac{\mathrm{e}^{-ks}}{\sqrt{1-\mathrm{e}^{-2s}}}\mathbb{E}\Bigg[(\Sigma^{-1/2}\mathbf{Z})_{i_l}\frac{\partial^{k-1} h}{\prod_{\stackrel{1\leq j\leq k}{j\not=l}}\partial w_{i_j}}(\mathrm{e}^{-s}\mathbf{w}+\sqrt{1-\mathrm{e}^{-2s}}\Sigma^{1/2}\mathbf{Z})\Bigg]\,\mathrm{d}s\\
&\quad\times(\mathbf{u}_1)_{i_1}\cdots(\mathbf{u}_k)_{i_k}\Bigg|\\
&=\Bigg|\int_0^{\infty}\frac{\mathrm{e}^{-ks}}{\sqrt{1-\mathrm{e}^{-2s}}}\mathbb{E}\Bigg[\sum_{i_l=1}^d(\mathbf{u}_l)_{i_l}(\Sigma^{-1/2}\mathbf{Z})_{i_l}\sum_{\{i_1,\ldots,i_k\}\setminus \{i_l\}}\bigg\{\frac{\partial^{k-1} h}{\prod_{\stackrel{1\leq j\leq k}{j\not=l}}\partial w_{i_j}}(\mathrm{e}^{-s}\mathbf{w}+\sqrt{1-\mathrm{e}^{-2s}}\Sigma^{1/2}\mathbf{Z})\\
&\quad\times (\mathbf{u}_1)_{i_1}\cdots(\mathbf{u}_{l-1})_{i_{l-1}}(\mathbf{u}_{l+1})_{i_{l+1}}\cdots(\mathbf{u}_k)_{i_k}  \bigg\}\Bigg]\,\mathrm{d}s\Bigg|\\
&\leq\sup_{|\mathbf{u}|=1}\mathbb{E}|\langle \mathbf{u},\Sigma^{-1/2}\mathbf{Z}\rangle|\sup_{\mathbf{w}\in\mathbb{R}^d}\|D^{k-1}h(\mathbf{w})\|_{op}\int_0^{\infty}\frac{\mathrm{e}^{-ks}}{\sqrt{1-\mathrm{e}^{-2s}}}\,\mathrm{d}s.
\end{align*} 
We now note that $\langle \mathbf{u},\Sigma^{-1/2}\mathbf{Z}\rangle\stackrel{\mathcal{D}}{=}\sum_{i,j=1}^d\tilde{\sigma}_{ij}u_iZ_j$, where $Z_1,\ldots,Z_d$ are independent standard normal random variables.  Hence, $\langle \mathbf{u},\Sigma^{-1/2}\mathbf{Z}\rangle$ follows the $N(0,\sum_{j=1}^d(\sum_{i=1}^d\tilde{\sigma}_{ij}u_i)^2)=N(0,|\Sigma^{-1/2}\mathbf{u}|^2)$ distribution, and so
\begin{equation*}\mathbb{E}|\langle \mathbf{u},\Sigma^{-1/2}\mathbf{Z}\rangle|=\sqrt{\frac{2}{\pi}}|\Sigma^{-1/2}\mathbf{u}|\leq \sqrt{\frac{2}{\pi}}\|\Sigma^{-1/2}\|_{op}|\mathbf{u}|=\sqrt{\frac{2}{\pi}}\|\Sigma^{-1/2}\|_{op}.
\end{equation*}
Applying the integral formula (\ref{intgamma}) now completes the proof of inequality (\ref{mkf}). 
\end{proof}

\begin{remark}\emph{If the covariance matrix $\Sigma$ is positive-definite, then, for any $i\in\{1,\ldots d\}$, we have
\[\Bigg[\sum_{j=1}^d\tilde{\sigma}_{ij}^2\Bigg]^{1/2}\leq \|\Sigma^{-1/2}\|_{op}.\]
To see this, consider the quantity $|\Sigma^{-1/2}\mathbf{u}|$, where $\mathbf{u}$ is a column vector with $i$-th entry $1$ and all other entries set to $0$.  Then, $|\Sigma^{-1/2}\mathbf{u}|=\sqrt{\sum_{j=1}^d\tilde{\sigma}_{ij}^2}$ and $|\Sigma^{-1/2}\mathbf{u}|\leq \|\Sigma^{-1/2}\|_{op}|\mathbf{u}|=\|\Sigma^{-1/2}\|_{op}$.}

\emph{It is also worth noting that if $\Sigma$ is a diagonal ($(\Sigma)_{ii}=\sigma_{ii}^2$ and $(\Sigma)_{ij}=0$ for $i\not=j$) then the bounds of Proposition (\ref{fingerscrossed}) simplify.  Indeed,
\[\Bigg[\sum_{j=1}^d\tilde{\sigma}_{ij}^2\Bigg]^{1/2}=\sigma_{ii}^{-1}\]
and
\begin{equation*}\|\Sigma^{-1/2}\|_{op}=\sup_{\mathbf{u}=1}|\Sigma^{-1/2}\mathbf{u}|=\sup_{\mathbf{u}=1}\Bigg[\sum_{j=1}^d\sigma_{ii}^{-2}u_i^2\Bigg]^{1/2}= \max_{1\leq i\leq d}\sigma_{ii}^{-1}.
\end{equation*}
}
\end{remark}

\begin{remark}\label{oiub} \emph{From the inequalities $\frac{\Gamma(x+\frac{1}{2})}{\Gamma(x+1)}>\frac{1}{\sqrt{x+\frac{1}{2}}}$ for $x>0$ (see the proof of Corollary 3.4 of Gaunt \cite{gaunt inequality}) and $\frac{\Gamma(x+\frac{1}{2})}{\Gamma(x+1)}<\frac{1}{\sqrt{x+\frac{1}{4}}}$ for $x>-\frac{1}{4}$ (see Elezovi\'c et al$.$ \cite{elezovic}), we have that  
\[\frac{1}{\sqrt{k}}<\frac{\Gamma(\frac{k}{2})}{\sqrt{2}\Gamma(\frac{k+1}{2})}<\frac{1}{\sqrt{k-\frac{1}{2}}}, \qquad k\geq 1.\]
Hence, for large $k$, bounds (\ref{charm}) and (\ref{mkf}) are of order $k^{-1/2}$; slower than the $k^{-1}$ rate of bounds (\ref{cheque}) and (\ref{meckesk}).\emph} 
\end{remark}

\begin{remark}\emph{The bounds (\ref{meckesk}) and (\ref{mkf}) give us the choice of two bounds on the derivatives of the solution to the multivariate normal Stein equation.  Bound (\ref{mkf}) allows us to impose weaker differentiability conditions on the test functions in multivariate approximation limit theorems obtained by use of Stein's method.  Bound (\ref{meckesk}) may, however, be preferable if computing $\|\Sigma^{-1/2}\|_{op}$ is difficult, or if we require bounds on derivatives of very large order (see the previous remark).  Analogous comments apply to bounds (\ref{cheque}) and (\ref{charm}).}
\end{remark}

We now specialise to the univariate case.  We begin by proving the following lemma. 

\begin{lemma}\label{zthird333467}Suppose that $h:\mathbb{R}\rightarrow\mathbb{R}$ is bounded.  Then the solution (\ref{hundred}) of the standard normal Stein equation (\ref{normal equation}) satisfies the bound
\begin{equation}\label{wfbound}\|wf'(w)\|\leq\|h-\Phi h\|.
\end{equation}
Suppose now that $h\in C_b^{n-1}(\mathbb{R})$, where $n\geq 2$.  Then 
\begin{equation}\label{f''bound} \|wf^{(k)}(w)\|\leq \|h^{(k-1)}\|, \qquad k=2,3,\ldots,n.
\end{equation}
\end{lemma}

\begin{proof}We first prove (\ref{wfbound}).  Suppose $w>0$.  From (\ref{hundred}) we have that
\begin{align*}|wf'(w)|&=\bigg|w\mathrm{e}^{w^2/2}\int_{w}^{\infty}[h(t)-\Phi h]\mathrm{e}^{-t^2/2}\,\mathrm{d}t\bigg|\\
&\leq\|h-\Phi h\|w\mathrm{e}^{w^2/2}\int_{w}^{\infty}\mathrm{e}^{-t^2/2}\,\mathrm{d}t\leq \|h-\Phi h\|,
\end{align*}
where we used the inequality $\int_w^{\infty}\mathrm{e}^{-t^2/2}\,\mathrm{d}t\leq w^{-1}\mathrm{e}^{-w^2/2}$ for $w>0$ (Chen et al$.$ \cite{chen}, p$.$ 37).  The proof for $w<0$ is similar, with the difference being that we use formula (\ref{hundred1}) for $f'(w)$ instead of formula (\ref{hundred}).

We now use inequality (\ref{wfbound}) to deduce inequality (\ref{f''bound}).  Suppose that $h\in C_b^{k-1}(\mathbb{R})$, where $k\geq 2$.  Substituting the integral formula (\ref{3x3min}) for $f'(w)$ into (\ref{wfbound}) and replacing $h-\Phi h$ by a differentiable function $g$ (note that $g'=h'$) yields the inequality
\begin{equation}\label{misdi}\bigg|w\int_0^{\infty}\!\int_{-\infty}^{\infty}\mathrm{e}^{-s}g'(\mathrm{e}^{-s}w+x\sqrt{1-\mathrm{e}^{-2s}})\phi(x)\,\mathrm{d}x\,\mathrm{d}s\bigg|\leq \|g\|,
\end{equation}
where $\phi(x)$ is the standard normal density.  From the integral formula (\ref{ff6ffvi}), we see that we require a bound on
\begin{equation}\label{yhn}|wf^{(k)}(w)|=\bigg|w\int_0^{\infty}\!\int_{-\infty}^{\infty}\frac{\mathrm{e}^{-ks}}{\sqrt{1-\mathrm{e}^{-2s}}}h^{(k-1)}(\mathrm{e}^{-s}w+x\sqrt{1-\mathrm{e}^{-2s}})x\phi(x)\,\mathrm{d}x\,\mathrm{d}s\bigg|, \qquad k\geq 2,
\end{equation}
and we deduce a bound for this quantity by using inequality (\ref{misdi}).

Now we use an approximation result that is given in Section 23.3 of Priestley \cite{priestley}.  If $\theta$ is an integrable function, then for any $\epsilon>0$ there exists a differentiable function $\varphi$ with compact support such that $|\int\theta-\int\varphi|\leq\int|\theta-\varphi|<\epsilon$.  Moreover, on examining the proof of this result (see Section 11.13 and Lemma 12.13), we see that the differentiable function $\varphi$ can be chosen such that $\|\varphi\|\leq\|\theta\|$.  Since $h^{(k-1)}$ is not assumed to be differentiable, we apply this result to approximate the integral (\ref{yhn}).  On doing so, we have that for every $\epsilon>0$ there exists a differentiable function $g$ with $\|g\|\leq\|h^{(k-1)}\|$ such that
\begin{align*}|wf^{(k)}(w)|&<\bigg|w\int_0^{\infty}\!\int_{-\infty}^{\infty}\frac{\mathrm{e}^{-ks}}{\sqrt{1-\mathrm{e}^{-2s}}}g(\mathrm{e}^{-s}w+x\sqrt{1-\mathrm{e}^{-2s}})x\phi(x)\,\mathrm{d}x\,\mathrm{d}s\bigg|+\epsilon \\
&=\bigg|w\int_0^{\infty}\!\int_{-\infty}^{\infty}\mathrm{e}^{-ks}g'(\mathrm{e}^{-s}w+x\sqrt{1-\mathrm{e}^{-2s}})\phi(x)\,\mathrm{d}x\,\mathrm{d}s\bigg|+\epsilon,
\end{align*} 
where we used integration by parts to obtain the equality.  Define $g_+'(t)\mathrel{\mathop:}=\mathrm{max}(g'(t),0)$ and $g_-'(t)\mathrel{\mathop:}=\mathrm{max}(-g'(t),0)$.  Since $g_+'{(k)}$ and $g_-'{(k)}$ have no sign changes and $\mathrm{e}^{-ks}< \mathrm{e}^{-s}$ for $k\geq 2$, we have
\begin{align*}|wf^{(k)}(w)|&< \max\bigg\{|w|\bigg|\int_0^{\infty}\!\int_{-\infty}^{\infty}\mathrm{e}^{-ks}g_+'(\mathrm{e}^{-s}w+x\sqrt{1-\mathrm{e}^{-2s}})\phi(x)\,\mathrm{d}x\,\mathrm{d}s\bigg|, \\
&\quad|w|\bigg|\int_0^{\infty}\!\int_{-\infty}^{\infty}\mathrm{e}^{-ks}g_-'(\mathrm{e}^{-s}w+x\sqrt{1-\mathrm{e}^{-2s}})\phi(x)\,\mathrm{d}x\,\mathrm{d}s\bigg|\bigg\}+\epsilon \\
&\leq \max\bigg\{|w|\bigg|\int_0^{\infty}\!\int_{-\infty}^{\infty}\mathrm{e}^{-s}g_+'(\mathrm{e}^{-s}w+x\sqrt{1-\mathrm{e}^{-2s}})\phi(x)\,\mathrm{d}x\,\mathrm{d}s\bigg|, \\
&\quad|w|\bigg|\int_0^{\infty}\!\int_{-\infty}^{\infty}\mathrm{e}^{-s}g_-'(\mathrm{e}^{-s}w+x\sqrt{1-\mathrm{e}^{-2s}})\phi(x)\,\mathrm{d}x\,\mathrm{d}s\bigg|\bigg\}+\epsilon \\
&\leq\max\{\|g\|,\:\|g\|\}+\epsilon\leq\|h^{(k-1)}\|+\epsilon,
\end{align*}
and to obtain the second to last inequality we used (\ref{misdi}) with $g'$ replaced by $g_+'$ and $g_-'$.  Letting $\epsilon\rightarrow 0$ completes the proof of inequality (\ref{f''bound}).
\end{proof}

With Lemma \ref{zthird333467} proved we are now able to establish the following bounds for the derivatives of the solution of the standard normal Stein equation. 

\begin{proposition}\label{snnewz}Suppose $h\in C_b^{n-2}(\mathbb{R})$, where $n\geq 3$.  Then the $k$-th order derivative of the solution (\ref{hundred}) of the standard normal Stein equation (\ref{normal equation}) satisfies the bound
\begin{equation}\label{asdert}\|f^{(k)}\|\leq 3\|h^{(k-2)}\|, \qquad k=3,4,\ldots,n.
\end{equation}
\end{proposition}

\begin{proof}The standard normal Stein equation is $f''(w)-wf'(w)=h(w)-\Phi h$.  By a straightforward induction on $k$ we have
\[f^{(k)}(w)=wf^{(k-1)}(w)+(k-2)f^{(k-2)}(w)+h^{(k-2)}(w).\]
Applying the triangle inequality and using inequalities (\ref{cheque}) and (\ref{f''bound}), we have, for every $w\in\mathbb{R}$,
\begin{align*}|f^{(k)}(w)|&\leq|wf^{(k-1)}(w)|+(k-2)\|f^{(k-2)}\|+\|h^{(k-2)}\| \\
&\leq \|h^{(k-2)}\|+(k-2)\cdot\frac{\|h^{(k-2)}\|}{k-2}+\|h^{(k-2)}\| \\
\label{60min}&=3\|h^{(k-2)}\|,
\end{align*}
as required.
\end{proof}

\begin{remark}\emph{Our bound (\ref{asdert}) for the univariate case involves derivatives of $h$ two degrees lower than $f$, whereas bound (\ref{charm}) for the multivariate case involves derivatives of $h$ just one degree lower than $f$.  This improvement comes from exploiting the fact that the standard normal Stein equation is a first order linear differential equation.  As was observed by Rai\v{c} \cite{raic clt} and Chatterjee and Meckes \cite{chatterjee 3}, this improvement is not possible in the multivariate case.  The example they considered was the function
\[h(w_1,w_2)=\max\{\min\{w_1,w_2\},0\}\]
for which the solution (\ref{mvnsoln}) of the multivariate normal Stein equation (\ref{mvn1}) is twice differentiable but $\frac{\partial^2f(\mathbf{w})}{\partial w_1\partial w_2}$ is not Lipschitz.}
\end{remark}

\section{Matching moments limit theorems}

We begin this section by establishing a general bound for the distance between the distribution of a standardised sum of independent random variables and the standard normal distribution.  We shall mostly concentrate on univariate limit theorems, although at the end of this section we note that the generalisation to vectors of independent random variables is straightforward.  

\begin{theorem}\label{123421}Let $X_1,X_2,\ldots X_n$ be independent random variables with $\mathbb{E}|X_i|^{p+1}<\infty$ for $1\leq i\leq n$.  For all $1\leq i\leq n$ and non-negative integers $k\leq p$, define $\epsilon_{i,k}=\mathbb{E}X_i^k-\mathbb{E}Z^k$.  Let $W_n=\frac{1}{\sqrt{n}}\sum_{i=1}^nX_i$.  Then, for all $h\in C_b^{p+1}(\mathbb{R})$, we have 
\begin{align}|\mathbb{E}h(W_n)-\Phi h|&\leq \frac{\|h'\|}{\sqrt{n}}\sum_{i=1}^n|\epsilon_{i,1}|+\sum_{i=1}^n\sum_{k=1}^{p-1}\frac{N_k}{k!n^{(k+1)/2}}|k\epsilon_{i,k-1}-\epsilon_{i,k+1}| \nonumber \\
\label{ffvi6}&\quad +\frac{N_p}{n^{(p+1)/2}}\sum_{i=1}^n\bigg(\frac{\mathbb{E}|X_i|^{p-1}}{(p-1)!}+\frac{\mathbb{E}|X_i|^{p+1}}{p!}\bigg),
\end{align}
where
\[N_k=\min\bigg\{3\|h^{(k-1)}\|,\:\frac{\Gamma(\frac{k+1}{2})\|h^{(k)}\|}{\sqrt{2}\Gamma(\frac{k}{2}+1)},\:\frac{\|h^{(k+1)}\|}{k+1}\bigg\}, \qquad 1\leq k\leq p.\]
\end{theorem}

\begin{proof}Throughout this proof we set $W:= W_n$.  We aim to bound $\mathbb{E}h(W)-\Phi h$, and do so by bounding $\mathbb{E}f''(W)-\mathbb{E}Wf'(W)$, where $f$ is the solution of the Stein equation (\ref{normal equation}).  We begin by letting $W^{(i)}=W-\frac{1}{\sqrt{n}}X_i$ and observing that $W^{(i)}$ and $X_i$ are independent.  Taylor expanding $f''(W)$ and $f'(W)$ about $W^{(i)}$ gives
\begin{align*}\mathbb{E}f''(W)-\mathbb{E}Wf'(W)&=\frac{1}{n}\sum_{i=1}^n\mathbb{E}f''(W)-\frac{1}{\sqrt{n}}\sum_{i=1}^n\mathbb{E}X_if'(W) \\
&=\sum_{i=1}^n\sum_{j=0}^{p-2}\frac{1}{j!n^{j/2+1}}\mathbb{E}X_i^j\mathbb{E}f^{(j+2)}(W^{(i)}) \\
&\quad-\sum_{i=1}^n\sum_{k=0}^{p-1}\frac{1}{k!n^{k/2+1/2}}\mathbb{E}X_i^{k+1}\mathbb{E}f^{(k+1)}(W^{(i)})+\sum_{i=1}^n(R_{1,i}+R_{2,i}),
\end{align*}
where
\begin{equation*}|R_{1,i}|\leq\frac{\mathbb{E}|X_i|^{p-1}\|f^{(p+1)}\|}{(p-1)!n^{(p+1)/2}} \qquad \mbox{and} \qquad |R_{2,i}|\leq\frac{\mathbb{E}|X_i|^{p+1}\|f^{(p+1)}\|}{p!n^{(p+1)/2}}.
\end{equation*}
Using independence and collecting terms, we can write this as
\begin{align*}\mathbb{E}f''(W)-\mathbb{E}Wf'(W)&=\sum_{i=1}^n\sum_{k=1}^{p-1}\frac{1}{k!n^{k/2+1/2}}[k\mathbb{E}X_i^{k-1}-\mathbb{E}X_i^{k+1}]\mathbb{E}f^{(k+1)}(W^{(i)})\\
&\quad-\frac{1}{\sqrt{n}}\sum_{i=1}^n\mathbb{E}X_i\mathbb{E}f'(W^{(i)})+\sum_{i=1}^n(R_{1,i}+R_{2,i}).
\end{align*}
Now, for even $k$ we have $k\mathbb{E}Z^{k-1}-\mathbb{E}Z^{k+1}=0$, and for odd $k$,
\[k\mathbb{E}Z^{k-1}-\mathbb{E}Z^{k+1}=k\cdot\frac{2^{(k-1)/2}\Gamma(\frac{k}{2})}{\sqrt{\pi}}-\frac{2^{(k+1)/2}\Gamma(\frac{k}{2}+1)}{\sqrt{\pi}}=0,\]
where we used that $\mathbb{E}|Z|^k=\frac{2^{k/2}\Gamma(\frac{k+1}{2})}{\sqrt{\pi}}$ (Winkelbauer \cite{winkelbauer}, formula 17) and the recurrence formula $\Gamma(x+1)=x\Gamma(x)$.  Therefore, recalling that $\epsilon_{i,k}=\mathbb{E}X_i^k-\mathbb{E}Z^k$, we have
\begin{align*}|\mathbb{E}f''(W)-\mathbb{E}Wf'(W)|&=\bigg|\sum_{i=1}^n\sum_{k=1}^{p-1}\frac{1}{k!n^{k/2+1/2}}[k\epsilon_{i,k-1}-\epsilon_{i,k}]\mathbb{E}f^{(k+1)}(W^{(i)})\\
&\quad-\frac{1}{\sqrt{n}}\sum_{i=1}^n\epsilon_{i,1}\mathbb{E}f'(W^{(i)})+\sum_{i=1}^n(R_{1,i}+R_{2,i})\bigg| \\
&\leq \sum_{i=1}^n\sum_{k=1}^{p-1}\frac{\|f^{(k+1)}\|}{k!n^{k/2+1/2}}|k\epsilon_{i,k-1}-\epsilon_{i,k}|\\
&\quad+\frac{\|f'\|}{\sqrt{n}}\sum_{i=1}^n|\epsilon_{i,1}|+\sum_{i=1}^n(|R_{1,i}|+|R_{2,i}|).
\end{align*}
Inequality (\ref{ffvi6}) now follows from using inequality (\ref{cheque}) to bound $\|f'\|$ and using inequalities (\ref{cheque}), (\ref{charm}) and (\ref{asdert}) to bound $\|f^{(j)}\|$ for $2\leq j\leq p+1$.
\end{proof}

The bound given in Theorem \ref{123421} is in terms of the quantities $\epsilon_{i,k}=\mathbb{E}X_i^k-\mathbb{E}Z^k$, and it is clear that if the $\epsilon_{i,k}$ are 'small' then the bound (\ref{ffvi6}) will also be 'small'.  In particular, if the first $p$ moments of the $X_i$ agree with the first $p$ moments of the standard normal distribution then we have a fast convergence rate:

\begin{corollary}\label{1234321ww}Let $X_1,X_2,\ldots,X_n$ be independent random variables with $\mathbb{E}X_i^k=\mathbb{E}Z^k$ for all $1\leq i\leq n$ and all positive integers $k\leq p$, and suppose that $\mathbb{E}|X_i|^{p+1}<\infty$ for $1\leq i\leq n$.  Let $W_n=\frac{1}{\sqrt{n}}\sum_{i=1}^nX_i$.  Then, for all $h\in C_b^{p+1}(\mathbb{R})$, we have
\begin{equation}\label{ffvi6yy}|\mathbb{E}h(W_n)-\Phi h|\leq \frac{N_p}{n^{(p+1)/2}}\sum_{i=1}^n\bigg(\frac{\mathbb{E}|X_i|^{p-1}}{(p-1)!}+\frac{\mathbb{E}|X_i|^{p+1}}{p!}\bigg),
\end{equation}
where $N_p$ is defined as in Theorem \ref{123421}.  With the additional assumption that the random variables are identically distributed, bound (\ref{ffvi6yy}) is of order $n^{-(p-1)/2}$ for large $n$.
\end{corollary}

\begin{proof}We have $\epsilon_{i,k}=0$ for all $1\leq i\leq n$ and non-negative integers $k\leq p$, and the result now follows immediately from Theorem \ref{123421}.
\end{proof}

Taking $n=1$ in Corollary \ref{1234321ww} gives the following bound on the distance between a probability distribution and the standard normal distribution when their first $p$ moments agree.  As one would expect, the distance tends to zero as $p\rightarrow\infty$.  

\begin{corollary}Let $X$ be a random variable with $\mathbb{E}X^k=\mathbb{E}Z^k$ for all positive integers $k\leq p$, and suppose that $\mathbb{E}|X|^{p+1}<\infty$.  Let $N_p$ be defined as in Theorem \ref{123421}.  Then, for all $h\in C_b^{p+1}(\mathbb{R})$, we have
\begin{equation*}|\mathbb{E}h(X)-\Phi h|\leq N_p\bigg(\frac{\mathbb{E}|X|^{p-1}}{(p-1)!}+\frac{\mathbb{E}|X|^{p+1}}{p!}\bigg).
\end{equation*}
\end{corollary}

\begin{remark}\emph{If we take $p=3$ in Corollary \ref{1234321ww}, we have a convergence rate of order $n^{-1}$, which is in agreement with the $n^{-1}$ rate for vanishing third moments obtained, through the use of zero bias couplings, by Goldstein and Reinert \cite{goldstein}.  Whilst their choice of coupling differs from ours, both bounds only differ by multiplicative constants and through the different inequalities used to bound the derivatives of the solution of the Stein equation.}
\end{remark}

\begin{remark}\emph{In general, we would except slower convergence rates in the Kolmogorov and Wasserstein metrics than the $n^{-(p-1)/2}$ rate of Corollary \ref{1234321ww}.  Consider the following example.  Suppose $Y_1,\ldots,Y_n$ are independent Bernoulli random variables with parameter $\frac{1}{2}$, and set $X_i=2(Y_i-\frac{1}{2})$ so that $\mathbb{E}X_i=\mathbb{E}X_i^3=0$ and $\mathbb{E}X_i^2=\mathbb{E}X_i^4=1$.  The first three moments of $X_i$ are equal to the first three moments of the standard normal distribution but the fourth moment differs, and so by Corollary \ref{1234321ww} we have a $n^{-1}$ bound for the quantity $|\mathbb{E}h(W_n)-\Phi h|$ for $h\in C_b^2(\mathbb{R})$.}

\emph{However, the Kolmogorov distance between $W_n$, which is a standardised Binomial distribution, and $Z$ is of order $n^{-1/2}$; see Hipp and Mattner \cite{hipp}.  The $n^{-1/2}$ rate is also optimal with respect to Wasserstein distance as can be seen as follows.  The Wasserstein distance between the distributions of $W_n$ and $Z$ is given by $d(\mathcal{L}(W_n),\mathcal{L}(Z))=\sup_{h\in\mathcal{H}}|\mathbb{E}h(W_n)-\mathbb{E}h(Z)|$, where $\mathcal{H}$ is the class of Lipschitz continuous functions on $\mathbb{R}$ with
Lipschitz constant not greater than 1.  Taking $h(x)=|x|$, we have $d(\mathcal{L}(W_n),\mathcal{L}(Z))\geq |\mathbb{E}|W_n|-\mathbb{E}|Z||$.  But
\begin{align*}|\mathbb{E}|W_{n+1}|-\mathbb{E}|W_n||&=\bigg|\frac{1}{\sqrt{n+1}}\mathbb{E}\bigg|\sum_{i=1}^{n+1}X_i\bigg|-\frac{1}{\sqrt{n}}\mathbb{E}\bigg|\sum_{i=1}^nX_i\bigg|\bigg|\\
&\geq\bigg|\frac{1}{\sqrt{n}}\bigg|\mathbb{E}\bigg|\sum_{i=1}^{n+1}X_i\bigg|-\mathbb{E}\bigg|\sum_{i=1}^{n}X_i\bigg|\bigg|-\bigg(\frac{1}{\sqrt{n}}-\frac{1}{\sqrt{n+1}}\bigg)\mathbb{E}\bigg|\sum_{i=1}^{n+1}X_i\bigg|\bigg|\\
&=\bigg|\frac{1}{2\sqrt{n}}-\frac{1}{\sqrt{n}}\bigg(1-\frac{1}{\sqrt{1+1/n}}\bigg)\mathbb{E}\bigg|\sum_{i=1}^{n+1}X_i\bigg|\bigg|=\frac{1}{2\sqrt{n}}+O(n^{-1}),
\end{align*}
where the inequality follows because $|a-b|\geq ||a|-|b||$ for $a,b\in\mathbb{R}$.  For the final assertion we used that for large $n$, $(1+1/n)^{-1/2}=1+O(n^{-1})$ and $\mathbb{E}|\sum_{i=1}^{n}X_i|=\sqrt{\frac{2n}{\pi}}+O(1)$ (see Chen and Shao \cite{chen9}, p$.$ 14).  It therefore follows that the quantity $|\mathbb{E}|W_n|-\mathbb{E}|Z||$ is bounded below by $Cn^{-1/2}$, where $C$ is a positive constant independent of $n$, and so the $n^{-1/2}$ rate is optimal.  Indeed, this argument can easily be extended to show that the Wasserstein distance between a standardised sum of independent, identically distributed discrete random variables and the standard normal distribution is always bounded below by $Kn^{-1/2}$, where $K$ is a positive constant independent of $n$, no matter how many moments are in agreement with those of the standard normal distribution.
}
\end{remark}

We now consider another application of Theorem \ref{123421}.  A well-known result (see Billingsley \cite{billingsley}, Example 30.1 and Theorem 30.2) is that if a random variable $X_n$ has moments of all order, and that $\lim_{n\rightarrow\infty}\mathbb{E}X_n^k=\mathbb{E}Z^k$ for $k=1,2,\ldots$, then $X_n$ converges in distribution to $Z$.  In the following theorem, we obtain a bound, for smooth test functions, on the distance between the distribution $X_n$ and the standard normal distribution.

\begin{theorem}\label{billley}Suppose $X_n$ is a random variable with moments of all order, and that $\mathbb{E}|X_n|^k=o(k!)$ for large $k$.  For all non-negative integers $k$, define $\epsilon_{k,n}=\mathbb{E}X_n^k-\mathbb{E}Z^k$.  Suppose further that there exist positive constants $\delta,$ $\alpha$ and $C$, which are independent of $k$ and $n$, such that
\begin{equation}\label{conddi}\|h^{(k)}\||\epsilon_{n,k}|\leq \frac{C}{n^{\alpha}}k^{-\delta}(k-1)! \quad \mbox{and} \quad \|h^{(k+2)}\||\epsilon_{n,k}|\leq \frac{C}{n^{\alpha}}k^{-\delta}(k-1)! 
\end{equation}
for all $k\geq 1$.  Then, for all $h\in C_b^{\infty}(\mathbb{R})$, 
\begin{equation}\label{xnnb}|\mathbb{E}h(X_n)-\Phi h|\leq\|h'\||\epsilon_{n,1}|+\sum_{k=1}^{\infty}\frac{\|h^{(k+1)}\|}{(k+1)!}|k\epsilon_{n,k-1}-\epsilon_{n,k+1}|\leq \frac{K}{n^{\alpha}},
\end{equation}
where $K$ is a constant which depends on $\delta$, $C$ and $h$ but is independent of $n$. 
\end{theorem}

\begin{proof}From Theorem \ref{123421}, we have that for all $p\geq 1$,
\begin{align}|\mathbb{E}h(X_n)-\Phi h|&\leq \|h'\||\epsilon_{i,1}|+\sum_{k=1}^{p-1}\frac{\|h^{(k+1)}\|}{(k+1)!}|k\epsilon_{n,k-1}-\epsilon_{n,k+1}|\nonumber  \\
\label{near1}&\quad+\|h^{(p+1)}\|\bigg(\frac{\mathbb{E}|X_n|^{p-1}}{(p-1)!}+\frac{\mathbb{E}|X_n|^{p+1}}{p!}\bigg).
\end{align}
To arrive at the bound (\ref{xnnb}) we must take the limit $p\rightarrow\infty$ on the right-hand side of (\ref{near1}).  We now justify taking this limit.  Using condition (\ref{conddi}), we have that for $k\geq 2$,
\begin{align*}\frac{\|h^{(k+1)}\|}{(k+1)!}|k\epsilon_{n,k-1}-\epsilon_{n,k+1}|&\leq \frac{\|h^{(k+1)}\|}{(k+1)!}\Big[k|\epsilon_{n,k-1}|+|\epsilon_{n,k+1}|\Big] \\
&\leq \frac{C}{n^{\alpha}}\bigg(\frac{1}{(k+1)(k-1)!k^{\delta}}+\frac{1}{(k+1)k^{\delta}}\bigg)\leq \frac{2C}{n^{\alpha}k^{1+\delta}}.
\end{align*}
The series $\sum_{k=1}^{\infty}\frac{1}{k^{1+\delta}}$ is convergent for positive $\delta$, and so, under the conditions of the theorem, the series
\[\sum_{k=1}^{\infty}\frac{\|h^{(k+1)}\|}{(k+1)!}|k\epsilon_{n,k-1}-\epsilon_{n,k+1}|\]
is convergent.  We also have that $\mathbb{E}|X_n|^k=o(k!)$ for large $k$, and so we may take the limit $p\rightarrow\infty$ on the right-hand side of (\ref{near1}):
\begin{equation*}|\mathbb{E}h(X_n)-\Phi h|\leq\|h'\||\epsilon_{n,1}|+\sum_{k=1}^{\infty}\frac{\|h^{(k+1)}\|}{(k+1)!}|k\epsilon_{n,k-1}-\epsilon_{n,k+1}|\leq  \frac{C}{n^{\alpha}}\bigg(\frac{3}{2}+2\sum_{k=1}^{\infty}\frac{1}{k^{1+\delta}}\bigg),
\end{equation*}
where we used that $\epsilon_{n,0}=0$, $\|h'\||\epsilon_{n,1}|\leq \frac{C}{n^{\alpha}}$ and $\|h''\||\epsilon_{n,2}|\leq \frac{C}{n^{\alpha}2^{\delta}}\leq\frac{C}{n^{\alpha}}$ to obtain the final inequality.  This completes the proof.
\end{proof}

\begin{remark}\emph{The condition $\mathbb{E}|X_n|^k=o(k!)$ for large $k$ given in Theorem \ref{billley} is very mild.  In practice, when applying Theorem \ref{billley}, we would consider random variables with moments very close to those of the normal distribution, and for large $k$, $\mathbb{E}|Z|^k=\frac{2^{k/2}\Gamma(\frac{k+1}{2})}{\sqrt{\pi}}\ll k!$.}
\end{remark}

We end this section by noting that the generalisation to vectors of independent random variables is straightforward.  Indeed, we have the following generalisation of Theorem \ref{123421}, from which generalisations of Corollary \ref{1234321ww} and Theorem \ref{billley} to vectors of independent random variables easily follow.

\begin{theorem}\label{lasthm}Let $X_{1,1},\ldots,X_{n,1},\ldots,X_{1,d},\ldots,X_{n,d}$ be independent random variables with finite $(p+1)$-th absolute moments.  For all $1\leq i\leq n$, $1\leq j\leq d$ and non-negative integers $k\leq p$, define $\epsilon_{i,j,k}=\mathbb{E}X_{i,j}^k-\mathbb{E}Z^k$.  Let $\mathbf{X}_i=(X_{i,1},\ldots,X_{i,d})^T$ and set $\mathbf{W}_n=\frac{1}{\sqrt{n}}\sum_{i=1}^n\mathbf{X}_i$.  Then, for all $h\in C_b^{p+1}(\mathbb{R})$, we have 
\begin{align*}|\mathbb{E}h(\mathbf{W}_n)-\mathbb{E}h(\mathbf{Z})|&\leq \sum_{i=1}^{n}\sum_{j=1}^d\frac{|\epsilon_{i,j,1}|}{\sqrt{n}}\bigg\|\frac{\partial h(\mathbf{w})}{\partial w_j}\bigg\|+\sum_{i=1}^{n}\sum_{j=1}^d\sum_{k=1}^{p-1}\frac{M_{j,k}}{k!n^{(k+1)/2}}|k\epsilon_{i,j,k-1}-\epsilon_{i,j,k+1}| \nonumber \\
\label{eqlast}&\quad +\sum_{i=1}^{n}\sum_{j=1}^d\frac{M_{j,p}}{n^{(p+1)/2}}\bigg(\frac{\mathbb{E}|X_{i,j}|^{p-1}}{(p-1)!}+\frac{\mathbb{E}|X_{i,j}|^{p+1}}{p!}\bigg),
\end{align*}
where
\[M_{j,k}=\min\bigg\{\:\frac{\Gamma(\frac{k+1}{2})}{\sqrt{2}\Gamma(\frac{k}{2}+1)} \bigg\|\frac{\partial^{k} h(\mathbf{w})}{\partial w_j^k}\bigg\|,\:\frac{1}{k+1}\bigg\|\frac{\partial^{k+1}h(\mathbf{w})}{\partial w_j^{k+1}}\bigg\|\bigg\}, \quad 1\leq j\leq d,\: 1\leq k\leq p.\]
\end{theorem}

\begin{proof}The Stein equation for the standard multivariate normal distribution is given by $\sum_{j=1}^d\left(\frac{\partial^2f}{\partial w_j^2}(\mathbf{w})-w_j\frac{\partial f}{\partial w_j}(\mathbf{w})\right)=h(\mathbf{w})-\mathbb{E}h(\mathbf{Z})$.  Evaluating both sides at $\mathbf{W}_n$, taking expectations and then carrying out Taylor expansions as we did in the proof of Theorem \ref{123421} yields the desired bound.  Here we used inequalities (\ref{cheque}) and (\ref{charm}) to bound the partial derivatives of the solution of the Stein equation.
\end{proof}

\section*{Acknowledgements}

During the course of this research the author was supported by an EPSRC DPhil Studentship, an EPSRC Doctoral Prize and EPSRC research grant AMRYO100.  The author would like thank Gesine Reinert for some productive discussions.  The author would also like to thank two anonymous referees for their helpful comments and suggestions, which helped me to prepare an improved manuscript.

\end{document}